\documentclass[12pt,reqno]{amsart}% F{O{L{D

\usepackage{amsfonts}
\usepackage{amsmath}
\usepackage{amssymb}
\usepackage{amsthm}
\usepackage{bbm}
\usepackage{bm}
\usepackage{datetime}
\usepackage{dsfont}
\usepackage{enumitem}
\usepackage{esvect}
\usepackage{extarrows}
\usepackage{fancyhdr}
\usepackage{float}
\usepackage[T1]{fontenc} % Output font encoding for international characters
\usepackage[top=3cm, bottom=2.5cm, left=2.5cm, right=2.5cm]{geometry}
\usepackage{graphicx}
\usepackage[bookmarks,breaklinks,colorlinks=true,linkcolor=blue,citecolor=blue,urlcolor=blue,menucolor=blue,pagebackref=true,linktoc=all,pdftex]{hyperref}
\usepackage{mathrsfs}
\usepackage{mathtools}
\usepackage{physics}
\usepackage{relsize}
\usepackage{thmtools}
\usepackage{tikz}
\usepackage[normalem]{ulem}%for writing a underlined text
\usepackage{url}
\usepackage{xcolor}

\graphicspath{{Images/}}

\makeatletter
\def\th@plain{%
    \thm@notefont{}% same as heading font
    \itshape% body font
}
\def\th@definition{%
    \thm@notefont{}% same as heading font
    \normalfont% body font
}
\makeatother

\newtheorem{theorem}{Theorem}[section]
\newtheorem{lemma}[theorem]{Lemma}
\newtheorem{proposition}[theorem]{Proposition}
\newtheorem{corollary}[theorem]{Corollary}

\theoremstyle{definition}

\theoremstyle{remark}
\newtheorem{remark}[theorem]{Remark}

\theoremstyle{theorem}

\theoremstyle{remark}

\newcommand{\mychi}{\raisebox{2pt}{$\chi$}}

\newcommand{\GL}{\mathrm{GL}}

\newcommand{\SO}{\mathrm{SO}}

\renewcommand{\d}{\dd}

\DeclareMathOperator*{\vol}{vol}

\numberwithin{equation}{section}

\begin{document}

% Title
\title{Quantitative rational approximation on spheres}% F{O{L{D

\author{Mahbub Alam}

\address{\textbf{Mahbub Alam} \\
    School of Mathematics,
    Tata Institute of Fundamental Research, Mumbai, India 400005}
\email{mahbub.dta@gmail.com}

\author{Anish Ghosh}

\address{\textbf{Anish Ghosh} \\
    School of Mathematics,
    Tata Institute of Fundamental Research, Mumbai, India 400005}
\email{ghosh@math.tifr.res.in}

\thanks{AG was supported by a Government of India, Department of Science and Technology, Swarnajayanti fellowship DST/SJF/MSA-01/2016--17, a CEFIPRA grant and a Matrics grant.
    MA and AG acknowledge support of the Department of Atomic Energy, Government of India, under project $12-R\&D-TFR-5.01-0500$.
    This work received support from a grant from the Infosys foundation.}

\date{}

% F}O}L}D

% Abstract
\begin{abstract}% F{O{L{D
    We prove a quantitative theorem for Diophantine approximation by rational points on spheres.
    Our results are valid for arbitrary unimodular lattices and we further prove `spiraling' results for the direction of approximates.
    These results are quantitative generalizations of the Khintchine-type theorem on spheres proved in~\cite{kleinbockmerrill15}.
\end{abstract}% F}O}L}D

\maketitle

% Section
\section{Introduction}\label{}% F{O{L{D
In this paper we prove quantitative results in intrinsic Diophantine approximation on spheres.
Intrinsic Diophantine approximation refers to a family of problems where one considers an algebraic variety $X$ with a dense set of rational points, and studies metric Diophantine approximation on $X$ with the aim of establishing analogues of the classical results in the theory.
This subject has some vintage, having been considered by Lang~\cite{Lang65}, and has seen considerable activity recently.
We refer the reader to~\cite{GGN13, GGN14, GK17} and~\cite{GGN20} for results on metric Diophantine approximation on homogeneous varieties of semisimple groups, to~\cite{kleinbockmerrill15, KS18, kleinmosh19, mosh16, sardari19, SW19, BGSV18} for results on spheres, and to~\cite{FKMS14, FMS18, PR21} for results on quadratic surfaces.
Intrinsic Diophantine approximation on spheres has received particular attention, both for its own sake and for connections to quantum gates as pointed out by Sarnak~\cite{Sarnak15} and to computer science~\cite{BahrdtSeybold17}.
Despite recent progress, there remain many open questions.
In particular, \emph{quantitative} Diophantine approximation on varieties poses a significant challenge.

Throughout this paper, all norms will represent Euclidean norms, rational points on the unit sphere $\mathbb{S}^n \subseteq \mathbb{R}^{n+1}$ centered at the origin will be represented as $\frac{\bm{p}}{q}$, where $q \in \mathbb{Z}_+$ and $\bm{p} \in \mathbb{Z}^{n+1}$ such that $\norm{\bm{p}} = q$ and $(\bm{p}, q) \in \mathbb{Z}^{n+2}$ is primitive.
We will denote by $\mathbb{Z}^{n+2}_{\mathrm{pr}}$ the set of primitive vectors in $\mathbb{Z}^{n+2}$.
% Note that we do not demand $(\bm{p}, q) \in \mathbb{Z}^{n+2}$ to be primitive.
Finally, elements of Euclidean spaces will always be thought of as column vectors even though we will write them as row vectors.
With these notational niceties taken care of, we can now recall a result of Kleinbock and Merrill.
% StatementofTheorem
\begin{theorem}[\cite{kleinbockmerrill15}, Theorem 1.1]\label{thmkmdioonsn}% F{O{L{D
    There exists a constant $c_n > 0$ such that for all $\alpha \in \mathbb{S}^n$ there exist infinitely many rationals $\frac{\bm{p}}{q} \in \mathbb{S}^n$ such that
    \[
        \norm{\alpha - \frac{\bm{p}}{q}} < \frac{c_n}{q}.
    \]
\end{theorem}% F}O}L}D

This result can be viewed as an analogue for spheres, of Dirichlet's classical theorem in Diophantine approximation.
We will prove a quantitative version of this result, in the sense that we give an asymptotic count (as $T \to \infty$) for the number of solutions to the above inequality with $q$ restricted between 1 and $\cosh T$.
For $\alpha \in \mathbb{S}^n$ and $T, c > 0$ set $N_{T, c}(\alpha)$ to be the number of solutions $(\bm{p}, q) \in \mathbb{Z}^{n+2}_{\mathrm{pr}}$ with $\frac{\bm{p}}{q} \in \mathbb{S}^n$ to
% UnnumberedAlignedEquation
\begin{equation*}\label{}
    \begin{split}
        \norm{\alpha - \frac{\bm{p}}{q}} < \frac{c}{q}, \\
        1 \leq q < \cosh T.
    \end{split}
\end{equation*}
We prove that
% StatementofTheorem
\begin{theorem}\label{thmintropaper}% F{O{L{D
    There exists a computable constant $\eta(c) > 0$ (\S\ref{ssecFTcvol}), depending only on $c$, such that for a.e.\ $\alpha \in \mathbb{S}^n$,
    \[
        \frac{N_{T, c}(\alpha)}{T} \to \eta(c) \ \text{as} \ T \to \infty.
    \]
\end{theorem}% F}O}L}D

Our result can be viewed as an analogue for spheres, of W.\ M.\ Schmidt's~\cite{schmidt60} classical counting result for Diophantine approximation in $\mathbb{R}^n$, although of course the latter holds for more general approximating functions and has an error term.
We also prove a spiraling result (Theorem~\ref{thmspiralingforsn}) about the direction of approximates $\alpha - \frac{\bm{p}}{q}$.
Such results about the distribution of approximates were considered first in~\cite{athreyaghoshtseng15, athreyaghoshtseng14} and also in~\cite{kleinbockshiweiss17}, for the case of Diophantine approximation in $\mathbb{R}^{n}$.
Finally, following~\cite{alamghosh20}, we also prove results for Diophantine approximation when $(\bm{p}, q)$ are elements of a `general lattice', not just of the integral lattice.\\

The only counting results regarding intrinsic Diophantine approximation on varieties that we are aware of are in~\cite{GGN20}.
In this paper, counting with error terms and discrepancy results are proved for intrinsic Diophantine approximation on semisimple group varieties using effective ergodic theorems for semisimple groups and quantitative duality principles.
However, even for the case in common, that of $S^3$, the results in the present paper are different from those in~\cite{GGN20} --- we count solutions for the `critical' or `Dirichlet' exponent in Diophantine approximation which is not the case in~\cite{GGN20}; however, in the range of exponents considered in~\cite{GGN20} stronger results are proved, including error terms as well as approximation using rational points with restricted denominators.
As far as we know, our counting results for the Dirichlet exponent in intrinsic Diophantine approximation constitute the first such results for \emph{any} variety.\\

Finally, our methods are completely different from~\cite{GGN20}.
We use the geometry of spheres, a dynamical reformulation of the counting problem (inspired by~\cite{kleinbockmerrill15}), Siegel's mean value formula and Birkhoff's ergodic theorem to prove our result.
The use of Birkhoff's ergodic theorem can be traced to~\cite{athreyaparrishtseng16} where it was used to give a new proof of (a special case of) W.\ M.\ Schmidt's theorem in classical Diophantine approximation for the Dirichlet exponent.
This idea was further developed in~\cite{alamghosh20} and extended to number fields as well as to arbitrary unimodular lattices.

% UnnumberedSubsection
\subsection*{Acknowledgments}% F{O{L{D
We thank the anonymous referees for their helpful remarks.
We thank Shucheng Yu for helpful discussions.

% F}O}L}D

% F}O}L}D
% Section
\section{Main results}\label{secmain}% F{O{L{D
We follow~\cite{kleinbockmerrill15}, where they explain a correspondence between intrinsic Diophantine approximation on $\mathbb{S}^n$ and dynamics on certain homogeneous spaces.

Let $\mathcal{Q} : \mathbb{R}^{n+2} \to \mathbb{R}$ be the quadratic form given by
% Equation
\begin{equation}\label{eqquadform}
    \mathcal{Q}(\bm{x}) := \sum_{i=1}^{n+1} x_{i}^{2} - {x}_{n+2}^{2}, \ \text{where} \ \bm{x} = (x_1, \ldots, x_{n+2}).
\end{equation}
Then we can embed $\mathbb{S}^n$ into the positive light-cone
\[
    \mathscr{L} := \{\bm{x} \in \mathbb{R}^{n+2} : \mathcal{Q}(\bm{x}) = 0, x_{n+2} \geq 0\}
\]
of $\mathcal{Q}$ via $\alpha \mapsto (\alpha, 1)$.
We will refer to this image also as $\mathbb{S}^n$.
Under this embedding primitive points $(\bm{p}, q) \in \Lambda_0 := \mathscr{L} \cap \mathbb{Z}^{n+2}_{\mathrm{pr}}$ and rational points $\frac{\bm{p}}{q} \in \mathbb{S}^n$ have a one-to-one correspondence, more specifically the line joining $(\bm{p}, q)$ and $\bm{0}$ intersects $\mathbb{S}^n$ at $\frac{\bm{p}}{q}$.
As~\cite{kleinbockmerrill15} notes in Lemma 2.4, `good approximates' $\frac{\bm{p}}{q}$ to $\alpha \in \mathbb{S}^n$ correspond to points $(\bm{p}, q) \in \Lambda_0$ which are close to the line through $(\alpha, 1)$.

% Remark
\begin{remark}\label{}% F{O{L{D
    For $\bm{x} \in \mathscr{L}$, $x_{n+2} \geq |x_i| ~\forall~ i = 1, \ldots, n+1$.
\end{remark}% F}O}L}D

Let $G := {\mathrm{SO}(\mathcal{Q})}^\circ \cong {\mathrm{SO}(n+1, 1)}^\circ$, i.e., $G$ is the connected component of identity of the group $\SO(\mathcal{Q})$.
By a \emph{lattice in} $\mathscr{L}$ we will mean a set of the form $g\Lambda_0$, where $g \in G$.
If $\Gamma$ denotes the stabilizer of $\Lambda_0$ in $G$, then $\Gamma$ is a lattice in $G$ containing $G_\mathbb{Z} := {\mathrm{SO}(\mathcal{Q})}_\mathbb{Z}^\circ$, the lattice of integer points of $G$, as a finite index subgroup.
The space of lattices in $\mathscr{L}$ can then be identified with the homogeneous space $X := G/\Gamma$.
Let $\mu$ be the left-invariant Haar measure on $G$ such that the induced unique left $G$-invariant measure on $X$, which we also denote by $\mu$, satisfies $\mu(X) = 1$.

Let ${\{\bm{u}_i\}}_{i=1}^{n+2}$ denote the standard Euclidean basis of $\mathbb{R}^{n+2}$, and $\bm{e}_1 := \bm{u}_{n+1} + \bm{u}_{n+2}$.
Let $K$ denote the subgroup of $G$ that preserves $\bm{u}_{n+2}$, i.e.,
\[
    K = \begin{pmatrix}
        \mathrm{SO}(n+1) &  \\
        \ & 1
    \end{pmatrix}
    \cong \mathrm{SO}(n+1).
\]
Then $K$ naturally acts on $\mathbb{R}^{n+1}$.
We endow $K$ with a unique left $K$-invariant probability Haar measure $\sigma$.

Note that $\mathbb{S}^n$ can be realized as a quotient of $K$, i.e., $\mathbb{S}^n \cong K/M$, where
\[
    M := \begin{pmatrix}
        \mathrm{SO}(n) &  \\
        \ & I_2
    \end{pmatrix} = \{k \in K : k\bm{u}_{n+1} = \bm{u}_{n+1}\}.
\]
Denote by $\d{m}$ the probability Haar measure on $M$.
This naturally endows $\mathbb{S}^n$ with a unique left $K$-invariant probability measure $\d{\widetilde{k}}$, and we note that there is a natural correspondence between full measure subsets of $K$ and full measure subsets of $\mathbb{S}^n$, under the projection $K \to K/M \cong \mathbb{S}^n$.

Let
\[
    g_t :=
    \renewcommand\arraystretch{1.5}
    \begin{pmatrix}
        I_n &  \\
        & \begin{matrix}
            \cosh t & -\sinh t \\
            -\sinh t & \cosh t
        \end{matrix}
    \end{pmatrix} \in G,
\]
and let
\[
    A := \{g_t : t \in \mathbb{R}\}.
\]
We endow $A$ with the natural measure $\d{t}$.

Setting $N$ to be the contracting horospherical subgroup associated to $\{g_t\}$, we have an Iwasawa decomposition of $G$, i.e., $G = NAK$.
Here
\[
    N = {\left\{u_{\bm{y}} =
            \renewcommand\arraystretch{2}
            \begin{pmatrix}
                I_n & -\bm{y} & \bm{y} \\
                \bm{y}^t & 1 - \frac{1}{2}\norm{\bm{y}}^2 & \frac{1}{2}\norm{\bm{y}}^2 \\
                \bm{y}^t & -\frac{1}{2}\norm{\bm{y}}^2 & 1 + \frac{1}{2}\norm{\bm{y}}^2
            \end{pmatrix} : \bm{y} =
            \renewcommand\arraystretch{1.3}
            \begin{pmatrix}
                y_1 \\
                \vdots \\
                y_n
            \end{pmatrix}\in \mathbb{R}^{n}\right\}},
\]
which is endowed with the natural measure $\d{\bm{y}}$.

It can be checked that under the coordinates $g = u_{\bm{y}}g_t k$ coming from the Iwasawa decomposition, $e^{-nt} \d{\bm{y}}\d{t}\d{\sigma(k)}$ is a left-invariant Haar measure on $G$, so by setting
\[
    \nu_\Gamma := \int_{G/\Gamma} e^{-nt} \d{\bm{y}}\d{t}\d{\sigma(k)}
\]
we see that
\[
    \d{\mu(g)} = \frac{1}{\nu_\Gamma} e^{-nt} \d{\bm{y}}\d{t}\d{\sigma(k)}.
\]

Note that $G$ acts on $\mathscr{L}$ by left-multiplication and the stabilizer of $\bm{e}_1$ is $Q := NM$, which is unimodular with a Haar measure $\d{q} := \d{\bm{y}}\d{m}$.
So this induces a unique left $G$-invariant measure $\lambda$ on $\mathscr{L}$.
Note that changing $\d{\bm{y}}\d{m}$ by a constant multiple changes $\lambda$ by a constant multiple.
Identifying $\mathscr{L} = G/Q$ with $K/M \times A$ we see that $\d{\widetilde{\lambda}} = e^{-nt} \d{t}\d{\widetilde{k}}$ is one such left $G$-invariant measure on $\mathscr{L}$.

For $f \in L^1(\mathscr{L}, \lambda)$, define $\widehat{f}$ (called the \emph{Siegel transform of $f$}) on $X$ by
% Equation
\begin{equation}\label{eqsiegeltrans}
    \widehat{f}(\Lambda) := \sum_{\bm{v} \in \Lambda \smallsetminus \{\bm{0}\}} f(\bm{v}).
\end{equation}
Then the \emph{Siegel integral formula} (Theorem~\ref{thmsmvt}) states that if $\lambda$ is suitably normalized then for any such $f$ we have
\[
    \int_{X} \widehat{f} \d{\mu} = \int_{\mathscr{L}} f \d{\lambda}.
\]
We call a function $f$ on $\mathscr{L}$, \emph{Riemann integrable} if $f$ is bounded with compact support and is continuous except on a set of $\lambda$-measure zero.

Throughout the rest of the paper for $k \in K$ denote $k\Lambda_0$ by $\Lambda_k$, also for fixed $\Delta \in X$ denote $k\Delta$ by $\Delta_k$.

% Subsection
\subsection{Diophantine approximation and Schmidt's theorem on $\mathbb{S}^n$}\label{}% F{O{L{D
We will derive Diophantine properties of points on $\mathbb{S}^n$ from the dynamics of $\{g_t\}$-action on $X$.
Let $T, c > 0$.
For $\alpha \in \mathbb{S}^n$ we are interested in $N_{T, c}(\alpha)$ and $N_{T, c}(\alpha; \Delta)$, the number of solutions $(\bm{p}, q) \in \Lambda_0$ and $(\bm{p}, q) \in \Delta$ respectively, to
% AlignedEquation
\begin{equation}\label{eqquandioonsn}
    \begin{split}
        \norm{\alpha - \frac{\bm{p}}{q}} < \frac{c}{q}, \\
        1 \leq q < \cosh T.
    \end{split}
\end{equation}
Define for $T > 0$
% UnnumberedGather
\begin{gather*}
    E_{T, c} := \{\bm{x} \in \mathscr{L} : 2x_{n+2}(x_{n+2} - x_{n+1}) < c^2, 1 \leq x_{n+2} < \cosh T\}, \\
    E_{T, c}(\Lambda) := E_{T, c} \cap \Lambda \ \text{for all} \ \Lambda \in X.
\end{gather*}

% Lemma
\begin{lemma}\label{}% F{O{L{D
    $(\bm{p}, q) \in \Delta$ is a solution to~\eqref{eqquandioonsn} if and only if $k(\bm{p}, q) \in E_{T, c} \ \forall k \in K$ with $k(\alpha, 1) = \bm{e}_1$.
\end{lemma}% F}O}L}D
% Proof
\begin{proof}% F{O{L{D
    Let $k(\alpha, 1) = \bm{e}_1$ and denote $k(\bm{p}, q)$ by $\bm{x} \in \mathscr{L}$.
    Note that $x_{n+2} = q$.
    Now,
    % UnnumberedAlign
    \begin{align*}
        \norm{\alpha - \frac{\bm{p}}{q}} < \frac{c}{q} &\iff \norm{q\alpha - \bm{p}} < c \\
        &\iff \norm{q(\alpha, 1) - (\bm{p}, q)} < c \\
        &\iff \norm{qk(\alpha, 1) - k(\bm{p}, q)} < c \quad (\text{applying $k \in K$ doesn't change the norm}) \\
        &\iff \norm{(x_1, x_2, \ldots, x_n, x_{n+1} - x_{n+2}, 0)} < c \\
        &\iff 2x_{n+2}(x_{n+2} - x_{n+1}) < c^2 \quad (\text{using the fact that $\bm{x} \in \mathscr{L}$}).
    \end{align*}
    Therefore $\bm{x} \in E_{T, c}$.
\end{proof}% F}O}L}D

% Remark
\begin{remark}\label{}% F{O{L{D
    % Enumerate
    % \renewcommand{\labelenumi}{\emph{\alph{enumi}.}}
    \begin{enumerate}[label= (\roman*),font=\normalfont,before=\normalfont]
        \item $k(\alpha, 1) = \bm{e}_1$ is equivalent to saying $k(\alpha) = \bm{u}_{n+1}$.
        \item $N_{T, c}(\alpha) = \#(E_{T, c}(\Lambda_k))$ and $N_{T, c}(\alpha; \Delta) = \#(E_{T, c}(\Delta_k))$, where $k(\alpha, 1) = \bm{e}_1$.
    \end{enumerate}
\end{remark}% F}O}L}D

Define for $T > 0$
% UnnumberedGather
\begin{gather*}
    F_{T, c} := \{\bm{x}\in \mathscr{L} : x_{n+2}^2 - x_{n+1}^2 < c^2, 1 \leq x_{n+2} + x_{n+1} < e^T\}, \\
    F_{T, c}(\Lambda) := F_{T, c} \cap \Lambda \ \text{for all} \ \Lambda \in X.
\end{gather*}
In a later section we will sandwich $E_{T, c}$ between sets of the form $F_{T, c}$, more explicitly we will show that for all sufficiently large $\ell \in \mathbb{Z}_+$, $\exists c_\ell \uparrow c$, compact sets $C_0, C_\ell \subseteq \mathscr{L}$ and a constant $r_0$ only depending on $c$ such that for all $T > r_0$
\[
    F_{T - r_0, c_\ell} \smallsetminus C_\ell \subseteq E_{T, c} \smallsetminus C_0 \subseteq F_{T+r_0, c}.
\]

% StatementofTheorem
\begin{theorem}[Special case of Schmidt's theorem for $\mathbb{S}^n$]\label{thmschforsn}% F{O{L{D
    For a.e.\ $\alpha \in \mathbb{S}^n$
    \[
        N_{T, c}(\alpha; \Delta) \thicksim |E_{T, c}| \ \text{as} \ T \to \infty.
    \]
    In particular for a.e.\ $\alpha \in \mathbb{S}^n$
    \[
        N_{T, c}(\alpha) \thicksim |E_{T, c}| \ \text{as} \ T \to \infty.
    \]
\end{theorem}% F}O}L}D

Here and throughout the rest of the paper $| \cdot |$ will denote the measure of a measurable subset of $\mathscr{L}$ with respect to $\lambda$ and $f(T) \thicksim g(T)$ as $T \to \infty$ will mean that $\frac{f(T)}{g(T)} \to 1$ as $T \to \infty$.
The above result, in view of the volume computations in \S\ref{secvolcomp}, proves Theorem \ref{thmintropaper}.
Theorem \ref{thmschforsn} can be derived immediately from the following.

% StatementofTheorem
\begin{theorem}\label{thmschforKlat}% F{O{L{D
    For a.e.\ $k \in K$
    % UnnumberedGather
    \begin{gather*}
        \#(F_{T, c}(\Delta_k)) \thicksim |F_{T, c}|, \\
        \#(E_{T, c}(\Delta_k)) \thicksim |E_{T, c}|.
    \end{gather*}
\end{theorem}% F}O}L}D

Theorems~\ref{thmschforsn} and~\ref{thmschforKlat} will be deduced from the following results concerning unimodular lattices in $X$.

% StatementofTheorem
\begin{theorem}\label{thmschforlat}% F{O{L{D
    For $\mu$-a.e.\ $\Lambda \in X$
    % Equation
    \begin{equation}\label{eqschforlat}
        \#(F_{T, c}(\Lambda)) \thicksim |F_{T, c}| \ \text{as} \ T \to \infty.
    \end{equation}
\end{theorem}% F}O}L}D

% StatementofCorollary
\begin{corollary}\label{corETcasympforlat}% F{O{L{D
    For any $\Lambda \in X$ satisfying~\eqref{eqschforlat} we have
    \[
        \#(E_{T, c}(\Lambda)) \thicksim |E_{T, c}| \ \text{as} \ T \to \infty.
    \]
\end{corollary}% F}O}L}D

% F}O}L}D
% Subsection
\subsection{Spiraling of approximations and spherical averages.}\label{}% F{O{L{D
We want to define `direction' of the approximate $\alpha - \frac{\bm{p}}{q}$ in~\eqref{eqquandioonsn}.
Let $\mathbb{S}^n_\bullet$ denote the doubly punctured sphere, i.e.,
\[
    \mathbb{S}^n_\bullet := \mathbb{S}^n \smallsetminus \{\pm\bm{u}_{n+1}\}.
\]
Define a function $\pi : \mathbb{S}^n_\bullet \to \mathbb{S}^{n-1} \subseteq \mathbb{R}^{n} \cong \mathcal{T}_{\bm{u}_{n+1}}(\mathbb{S}^n) \subseteq \mathbb{R}^{n+1}$ as follows.
For $\bm{x} = (x_1, \ldots, x_{n+1})$ in $\mathbb{S}^n_\bullet$ project $\bm{x}$ onto the tangent space of $\mathbb{S}^n$ at $\bm{u}_{n+1}$, denoted by $\mathcal{T}_{\bm{u}_{n+1}}(\mathbb{S}^n)$, to get $(x_1, \ldots, x_n, 0)$, which is non-zero.
We project this point onto $\mathbb{S}^{n-1} \subseteq \mathcal{T}_{\bm{u}_{n+1}}(\mathbb{S}^n)$ to define $\pi(\bm{x})$, i.e.,
% Equation
\begin{equation}\label{eqpionsn}
    \pi(\bm{x}) := \frac{(x_1, \ldots, x_n)}{\norm{(x_1, \ldots, x_n)}}.
\end{equation}
$\pi(\bm{x})$ denotes the direction of $\bm{x} \in \mathbb{S}^n_\bullet$ away from $\bm{u}_{n+1}$.
Inspired by~\eqref{eqpionsn} we define $\pi$ on $\mathscr{L}_\bullet := \mathscr{L} \smallsetminus \{\bm{x} \in \mathscr{L} : |x_{n+1}| = x_{n+2}\}$ as
% Equation
\begin{equation}\label{eqpionL}
    \pi(\bm{x}) := \frac{(x_1, \ldots, x_n)}{\norm{(x_1, \ldots, x_n)}}.
\end{equation}
Note that $\pi$ is defined a.e.\ on $\mathscr{L}$.

For a.e.\ $\alpha \in \mathbb{S}^n$, for any $(\bm{p}, q) \in \Delta$, $\norm{\bm{u}_{n+1} - \frac{k(\bm{p})}{q}} = \norm{\alpha - \frac{\bm{p}}{q}} \neq 0, \ \forall k \in K$ with $k(\alpha) = \bm{u}_{n+1}$; hence for $(\bm{p}, q)$ satisfying~\eqref{eqquandioonsn} with $q > c/2$ we have $\frac{k(\bm{p})}{q} \in \mathbb{S}^n_\bullet$ and $k(\bm{p}, q) \in \mathscr{L}_\bullet$ irrespective of which $k$ satisfying $k(\alpha) = \bm{u}_{n+1}$ we pick.
Hence $\pi\left(\frac{k(\bm{p})}{q}\right) = \pi(k(\bm{p}, q))$ is well-defined.
For $\frac{\bm{p}}{q} \neq \alpha$ the `direction' of $\frac{\bm{p}}{q}$ away from $\alpha$ corresponds to the `direction' of $\frac{k(\bm{p})}{q}$ away from $\bm{u}_{n+1}$.

For fixed $k \in K$ with $k(\alpha) = \bm{u}_{n+1}$ and $A \subseteq \mathbb{S}^{n-1}$ with boundary measure zero, let $N_{T, c, A}(\alpha, k)$ and $N_{T, c, A}(\alpha, k; \Delta)$ denote the number of solutions $(\bm{p}, q) \in \Lambda_0$ and $(\bm{p}, q) \in \Delta$ respectively to
% AlignedEquation
\begin{equation}\label{eqquandiosnspiraling}
    \begin{split}
        \norm{\alpha - \frac{\bm{p}}{q}} < \frac{c}{q}, \\
        1 \leq q < \cosh T, \\
        \pi(k(\bm{p}, q)) \in A.
    \end{split}
\end{equation}
Define
% UnnumberedGather
\begin{gather*}
    E_{T, c, A} := \{\bm{x} \in E_{T, c} : \pi(\bm{x}) \in A\}, \\
    E_{T, c, A}(\Lambda) := E_{T, c, A} \cap \Lambda, \\
    F_{T, c, A} := \{\bm{x} \in F_{T, c} : \pi(\bm{x}) \in A\}, \\
    F_{T, c, A}(\Lambda) := F_{T, c, A} \cap \Lambda.
\end{gather*}
Then
% UnnumberedGather
\begin{gather*}
    N_{T, c, A}(\alpha, k) = \#(E_{T, c, A}(\Lambda_k)), \\
    N_{T, c, A}(\alpha, k; \Delta) = \#(E_{T, c, A}(\Delta_k)), \\
\end{gather*}

% StatementofTheorem
\begin{theorem}[Spiraling for $\mathbb{S}^n$]\label{thmspiralingforsn}% F{O{L{D
    Let $A \subseteq \mathbb{S}^{n-1}$ be a measurable set with boundary measure zero, then for a.e.\ $\alpha \in \mathbb{S}^n$ there exists $k \in K$ with $k(\alpha) = \bm{u}_{n+1}$ such that
    \[
        N_{T, c, A}(\alpha, k; \Delta) = \#(E_{T, c, A}(\Delta_k)) \thicksim |E_{T, c, A}| \ \text{as} \ T \to \infty.
    \]
    In particular, for a.e.\ $\alpha \in \mathbb{S}^n$ there exists $k \in K$ with $k(\alpha) = \bm{u}_{n+1}$ such that
    \[
        N_{T, c, A}(\alpha, k) = \#(E_{T, c, A}(\Lambda_k)) \thicksim |E_{T, c, A}| \ \text{as} \ T \to \infty.
    \]
\end{theorem}% F}O}L}D

We prove spiraling for $\mathbb{S}^n$ with the help of the following theorem:

% StatementofTheorem
\begin{theorem}\label{thmspiralingwithFE}% F{O{L{D
    Let $A \subseteq \mathbb{S}^{n-1}$ be a measurable set with boundary measure zero, then for a.e.\ $k \in K$
    % UnnumberedGather
    \begin{gather*}
        \#(F_{T, c, A}(\Delta_k)) \thicksim |F_{T, c, A}| \ \text{as} \ T \to \infty, \\
        \#(E_{T, c, A}(\Delta_k)) \thicksim |E_{T, c, A}| \ \text{as} \ T \to \infty.
    \end{gather*}
\end{theorem}% F}O}L}D

% F}O}L}D

% F}O}L}D
% Section
\section{Sandwiching $E_{T, c, A}$, and volume of $F_{T, c, A}$ and $E_{T, c, A}$}\label{secvolcomp}% F{O{L{D
% Subsection
\subsection{Sandwiching $E_{T, c, A}$}\label{}% F{O{L{D
Let $C_0 = \{\bm{x} \in \mathscr{L} : x_{n+2} \leq c^2+1\}$ and $\bm{x} \in E_{T, c, A} \smallsetminus C_0$.
Then
\[
    2x_{n+2}(x_{n+2} - x_{n+1}) < c^2 \ \text{and} \ 1 \leq x_{n+2} \implies x_{n+2} - x_{n+1} \leq c^2,
\]
hence
\[
    x_{n+2} + x_{n+1} \geq 2x_{n+2} - c^2 \geq 1.
\]
Also
\[
    x_{n+1} \leq x_{n+2} \implies x_{n+2} + x_{n+1} \leq 2x_{n+2} < 2 \cosh T < e^{T+r_0} \ \text{for some constant} \ r_0.
\]
Finally
\[
    x_{n+1} \leq x_{n+2} \implies x_{n+2}^2 - x_{n+1}^2 = (x_{n+2} + x_{n+1})(x_{n+2} - x_{n+1}) \leq 2x_{n+2}(x_{n+2} - x_{n+1}) < c^2.
\]
Therefore
\[
    E_{T, c, A} \smallsetminus C_0 \subseteq F_{T+r_0, c, A}.
\]

For integers $\ell > c^2 + 1$, let $c_\ell := c \cdot {\left(1 - \frac{c^2}{2\ell}\right)}^{1/2}$ and $C_\ell := \{\bm{x} \in \mathscr{L} : x_{n+2} \leq \ell\}$.
Clearly $c_\ell \uparrow c$ as $\ell \to \infty$.
Let $T \geq r_0$ and $\bm{x} \in F_{T - r_0, c_\ell, A} \smallsetminus C_\ell$.
Then
\[
    x_{n+2}^2 - x_{n+1}^{2} < c_{\ell}^{2} \ \text{and} \ 1 \leq x_{n+2} + x_{n+1} \implies x_{n+2} - x_{n+1} < c^2,
\]
hence
\[
    2x_{n+2} < x_{n+2} + x_{n+1} + c^2 < e^{T - r_0} + c^2 \leq 2\cosh (T - r_0) + c^2 \leq 2\cosh T
\]
($r_0$ could be chosen to satisfy both these inequalities).
From $x_{n+2} - x_{n+1} < c^2$ we again get that
\[
    x_{n+2} + x_{n+1} \geq 2x_{n+2} - c^2 \geq 2x_{n+2}{\left(1 - \frac{c^2}{2\ell}\right)},
\]
hence
\[
    2x_{n+2}(x_{n+2} - x_{n+1}) \leq {\left(1 - \frac{c^2}{2\ell}\right)}^{-1}(x_{n+2}^2 - x_{n+1}^2) < c_\ell^2 \cdot {\left(1 - \frac{c^2}{2\ell}\right)}^{-1} = c^2.
\]
Therefore for all sufficiently large integers $\ell$
% Equation
\begin{equation}\label{eqETcAsandwiched}
    F_{T - r_0, c_\ell, A} \smallsetminus C_\ell \subseteq E_{T, c, A} \smallsetminus C_0 \subseteq F_{T + r_0, c, A}.
\end{equation}

% F}O}L}D
% Subsection
\subsection{The volume of $F_{T, c, A}$}\label{ssecFTcvol}% F{O{L{D
We prove that $|F_{T, c, A}|$ is proportional to $T$.
It suffices to show that
% Equation
\begin{equation}\label{eqFTcAvol}
    r \cdot |F_{T, c, A}| =  T \cdot |F_{r, c, A}|.
\end{equation}
For $T, r > 0$ such that $T/r = \ell \in \mathbb{Z}_+$, we have
\[
    F_{T, c, A} = \bigsqcup_{j = 0}^{\ell - 1} g_{-rj}(F_{r, c, A}),
\]
and hence~\eqref{eqFTcAvol} follows for $T/r \in \mathbb{Z}_+$, since $\{g_t\}$ preserves $\lambda$.
From this one deduces~\eqref{eqFTcAvol} for $T/r \in \mathbb{Q}$.
Finally for arbitrary $T, r > 0$ let $T_1, T_2 > 0$ be such that $T_1 < T < T_2$ and $T_1/r, T_2/r \in \mathbb{Q}$.
Then
% UnnumberedAlign
\begin{gather*}
    F_{T_1, c, A} \subseteq F_{T, c, A} \subseteq F_{T_2, c, A} \\
    \implies T_1 \cdot |F_{r, c, A}| \leq r \cdot |F_{T, c, A}| \leq T_2 \cdot |F_{r, c, A}|.
\end{gather*}
Hence we get~\eqref{eqFTcAvol} for all $T, r > 0$ by taking limits.

Note that $|F_{T, c, A}| = |F_{T, c}| \cdot \vol(A)$, where $\vol$ denotes the normalized probability measure on $\mathbb{S}^{n-1}$, and $\eta(c) = |F_{1, c}|$ in Theorem~\ref{thmintropaper}.

% F}O}L}D
% Subsection
\subsection{The volume of $E_{T, c, A}$}\label{ssecasympETcvol}% F{O{L{D
We want to prove that
% Equation
\begin{equation}\label{eqETcAasymp}
    |E_{T, c, A}| \thicksim |F_{T, c, A}| \ \text{as} \ T \to \infty.
\end{equation}
From~\eqref{eqETcAsandwiched} and~\eqref{eqFTcAvol} it follows that
% Equation
\begin{equation}\label{eqETcAvol1}
    \limsup_{T \to \infty} \frac{|E_{T, c, A}|}{T} \leq \limsup_{T \to \infty} \frac{|F_{T, c, A}|}{T} = |F_{1, c, A}|.
\end{equation}
Also from~\eqref{eqETcAsandwiched} and~\eqref{eqFTcAvol} it follows that for all integers $\ell$ sufficiently large
% Equation
\begin{equation}\label{eqETcAvol2}
    \liminf_{T \to \infty} \frac{|E_{T, c, A}|}{T} \geq \liminf_{T \to \infty} \frac{|F_{T - r_0, c_\ell, A}|}{T} = |F_{1, c_\ell, A}|.
\end{equation}
Since $c \mapsto |F_{1, c, A}|$ is a continuous function and $c_\ell \uparrow c$ as $\ell \to \infty$, \eqref{eqETcAvol2} implies that
% Equation
\begin{equation}\label{eqETcAvol3}
    \liminf_{T \to \infty} \frac{|E_{T, c, A}|}{T} \geq |F_{1, c, A}|.
\end{equation}
From~\eqref{eqFTcAvol},~\eqref{eqETcAvol1} and~\eqref{eqETcAvol3},~\eqref{eqETcAasymp} follows.
\newline

% F}O}L}D

All results in this section are true for whole of $E_{T, c}$ and $F_{T, c}$ as well.
Although
\[
    E_{T, c, \mathbb{S}^{n-1}} \subsetneq E_{T, c} \qquad \text{and} \qquad  F_{T, c, \mathbb{S}^{n-1}} \subsetneq F_{T, c},
\]
these sets are equal modulo measure zero sets.

% F}O}L}D
% Section
\section{Proof of Theorem~\ref{thmschforlat} and Corollary~\ref{corETcasympforlat}}\label{secpfsforlats}% F{O{L{D
We will need an analogue of Siegel's mean value theorem for the positive light-cone $\mathscr{L}$, this is a theorem about the average number of lattice points in a given subset of $\mathscr{L}$.

% StatementofTheorem
\begin{theorem}[Siegel's integral formula for the light-cone]\label{thmsmvt}% F{O{L{D
    The measure $\lambda$ can be suitably normalized so that, $\forall f \in L^1(\mathscr{L}, \lambda)$
    \[
        \int_{X} \widehat{f} \d{\mu} = \int_{\mathscr{L}} f \d{\lambda}.
    \]
\end{theorem}% F}O}L}D
% Remark
% \begin{remark}\label{}% F{O{L{D
%     The Theorem above follows from work of Shucheng Yu \cite{shucheng18}.
%     Using Prop.\ 6.2.2 in \cite{shucheng18} one can write the Siegel transform above as a sum of incomplete Eisenstein series (defined at different cusps); see the equation in Lemma 6.2.4.
%     Then one can compute the first moment of these incomplete Eisenstein series by a standard unfolding argument as in equation 4.1.4 in loc.\ cit.
% \end{remark}% F}O}L}D

Note that the Theorem above, as a consequence, says that $\widehat{f} \in L^1(X, \mu)$. The Theorem above is due to Shucheng Yu and can be assembled from his thesis \cite{shucheng18}. In order to keep this paper self-contained, we provide a proof below by gathering all the necessary facts from \cite{shucheng18}.
% UnnumberedSubsection
\subsection*{Preliminaries for the above proof:}\label{ssecyudeets}% F{O{L{D
An important ingredient of the proof is to realize $\widehat{f}$ as the sum of incomplete Eisenstein series for certain dilations of the function $f$.
Each of these incomplete Eisenstein series is associated to a cusp of $G_\mathbb{Z}$.
%We add some details from \cite{shucheng18} on these matters.

% \textbf{Discussion on cusps of $G_\mathbb{Z}$:}
We say that two parabolic subgroups of $G$ are $G_\mathbb{Z}$-equivalent if they are conjugate under $G_\mathbb{Z}$.
A \emph{cusp} of $G_\mathbb{Z}$ is a $G_\mathbb{Z}$-equivalent class of parabolic subgroups of $G$ whose unipotent radicals intersect $G_\mathbb{Z}$ nontrivially.
There are finitely many such parabolic subgroups, i.e., cusps, and we denote them by $P_1, \ldots, P_h$.
For example, $P = NAM$ is the unique parabolic subgroup containing $N$ as its unipotent radical associated to the Iwasawa decomposition of $G$ discussed in \S\ref{secmain}.
This decomposition $P = NAM$ is a Langlands decomposition of $P$.

Note that $G$ naturally acts on the space of parabolic subgroups by conjugation: $g \in G$ sends a parabolic subgroup $P'$ to $gP'g^{-1}$.
Under this action $P$ is its own stabilizer, so the space of parabolic subgroups can be identified with the homogeneous space $G/P$. Moreover,  
$G/P$ can further be identified with $K/M$.
Thus we can take $\{\xi_1, \ldots, \xi_h\} \subseteq K$ such that $\xi_jP_j\xi_j^{-1} = P$ for all $1 \leq j \leq h$.
Each $P_j$ thus has a Langlands decomposition $P_j = N_jA_jM_j$, where $N_j = \xi_j^{-1}N\xi_j$, $A_j = \xi_j^{-1}A\xi_j$ and $M_j = \xi_j^{-1}M\xi_j$.
For each $1 \leq j \leq h$ let $Q_j := N_jM_j$, $\Gamma_{P_j} := G_\mathbb{Z} \cap P_j$ and $\Gamma_{N_j} := G_\mathbb{Z} \cap N_j$.
By definition of cusps, each $\Gamma_{N_j}$ is nontrivial.

We now record a few facts about these groups.
Since $G_\mathbb{Z}$ is a lattice in $G$, $\Gamma_{N_j}$ is a lattice in $N_j$.
Furthermore, the discreteness of $\Gamma$ and nontriviality of $\Gamma_{N_j}$ implies that $\Gamma_{P_j} \subseteq Q_j$, since otherwise there will be a sequence of non-identity elements of $\Gamma$ converging to the identity element.
After identifying $N_j$ with $\mathbb{R}^{d}$ and $\Gamma_{N_j}$ with a lattice $L$ in $\mathbb{R}^{d}$, the action of $\Gamma_{P_j}$ on $N_j$ and $\Gamma_{N_j}$ by conjugation gives us an injection
\[
    \Gamma_{P_j}/\Gamma_{N_j} \hookrightarrow \SO(\mathbb{R}^{n}) \cap \GL(L).
\]
Thus the index of $\Gamma_{N_j}$ in $\Gamma_{P_j}$, $[\Gamma_{P_j} : \Gamma_{N_j}]$, is finite.

Using conjugation, we see that $\displaystyle\int_{N/\xi_j\Gamma_{N_j}\xi_j^{-1}} \d{\bm{y}}$ is the same as $\vol(N_j/\Gamma_{N_j})$, the covolume of the lattice $\Gamma_{N_j}$ in $N_j$.
Similarly,
\[
    \omega_j := \int_{Q/\xi_j\Gamma_{P_j}\xi_j^{-1}} \d{q}
\]
denotes the covolume of $\Gamma_{P_j}$ inside $Q_j$ (note that $\Gamma_{P_j} \subseteq Q_j$ is co-compact).
Using the decomposition $Q_j/\Gamma_{N_j} = N_j/\Gamma_{N_j} \times M_j$ and the fact that $\d{m}$ is the probability measure on $M$, we remark that
\[
    \omega_j = \frac{1}{[\Gamma_{P_j} : \Gamma_{N_j}]} \int_{Q/\xi_j\Gamma_{N_j}\xi_j^{-1}} \d{q} = \frac{\vol(N_j/\Gamma_{N_j})}{[\Gamma_{P_j} : \Gamma_{N_j}]}.
\]

We recall from \S\ref{secmain} that $G$ acts on $\mathscr{L}$ by left-multiplication with the stabilizer of $\bm{e}_1$ being $Q = NM$, and that we have the identification of $\mathscr{L}$ with $G/Q$ using the map $gQ \mapsto g\bm{e}_1$.
Define the incomplete Eisenstein series for a bounded compactly supported function on $\mathscr{L} = G/Q$ as
\[
    \Theta_{f}^{j}(g) := \sum_{\gamma \in G_\mathbb{Z}/\Gamma_{P_j}} f(g\gamma\xi_j^{-1}).
\]
Then $\Theta_{f}^{j}$ can be shown to be well-defined as a sum over the quotient space $G_\mathbb{Z}/\Gamma_{P_j}$ due to the right $Q$-invariance of $f$.
Furthermore, this definition is independent of the choice of $\{\xi_j\}$.
See \cite[page 28]{shucheng18} for details.
We note that these $\Theta_{f}^{j}$ induce functions on $G/G_\mathbb{Z}$.
We now restate a lemma and a remark from \cite{shucheng18} that go into the proof of Theorem \ref{thmsmvt}.
Let $f = \mychi_{B}$ be the characteristic function of a bounded measurable subset $B$ of $\mathscr{L}$.
Recall from \S\ref{secmain} the definition of $\nu_\Gamma$ and that $\Lambda_0 = \mathscr{L} \cap \mathbb{Z}_{\operatorname{pr}}^{n+2}$.

% StatementofLemma
\begin{lemma}[Lemma 6.2.4 of \cite{shucheng18}]\label{lemyu1}% F{O{L{D
    There exists some positive constants $\kappa_1 = 1 , \kappa_2, \ldots, \kappa_h$ independent of $B$ such that for any $\Lambda = g\Lambda_0 \in G/G_\mathbb{Z}$ 
    \[
        \widehat{f}(\Lambda) = \sum_{j = 1}^{h} \Theta_{f_j}^{j}(g),
    \]
    where $f_j = \mychi_{\kappa_jB}$ the characteristic function of the dilation $\kappa_jB$ of $B$.
\end{lemma}% F}O}L}D

% StatementofLemma
\begin{lemma}[Remark 4.1.3 of \cite{shucheng18}]\label{lemyu2}% F{O{L{D
    For any bounded compactly supported function $f$ on $\mathscr{L}$ the first moment of $\Theta_{f}^{j}$ is given as follows:
    \[
        \int_{G/G_\mathbb{Z}} \Theta_{f}^{j}(g) \d{\mu(g)} = \frac{\omega_j}{\nu_\Gamma} \int_{\mathscr{L}} f \d{\widetilde{\lambda}}.
    \]
\end{lemma}% F}O}L}D

This ends the preliminary discussion needed for the proof of Theorem \ref{thmsmvt}.

% F}O}L}D

\vspace{1.5em}

% Proof
\begin{proof}[Proof of Theorem \ref{thmsmvt}]% F{O{L{D
    % The Theorem above follows from the work of Shucheng Yu \cite{shucheng18}.
    % Here we shall put together the relevant information from \cite{shucheng18} to sketch a proof for the case when $f = \mychi_B$ is the characteristic function of a bounded subset $B$ of $\mathscr{L}$; the general case follows from this.
    Here we shall provide a proof for the case when $f = \mychi_B$ is the characteristic function of a bounded measurable subset $B$ of $\mathscr{L}$; the general case follows from this.
    % To make the paper more self-contained we add details from \cite{shucheng18} in the next subsection (\S\ref{ssecyudeets}).

    % We shall follow the notation from \cite{shucheng18}.
    % For $g \in G$ let
    % \[
    %     \widehat{f}_{\mathrm{pr}}(g\Lambda_0) := \sum_{\substack{\bm{v} \in g\Lambda_0 \smallsetminus \{\bm{0}\} \\ \bm{v} \ \text{primitive}}} f(\bm{v}).
    % \]
    From Lemma \ref{lemyu1} we see that
    \[
        \int_{G/G_\mathbb{Z}} \widehat{f} \d{\mu} = \sum_{ j = 1}^{h} \int_{G/G_\mathbb{Z}} \Theta^{j}_{f_j}(g) \d{\mu}.
    \]
    % where $h$ is the number of cusps of $G_\mathbb{Z}$, $f_j = \mychi_{\kappa_jB}$ is the characteristic function of the set $\kappa_jB$ for certain positive constants $\kappa_1 = 1, \kappa_2, \ldots, \kappa_h$ independent of $B$, and $\Theta^{j}_{f_j}$ is the incomplete Eisenstein series defined in Chapter 4, page 44 of \cite{shucheng18}.
    % attached to the function $f_j$ defined at a cusp of $G_\mathbb{Z}$.

    Applying Lemma \ref{lemyu2} we find
    \[
        \int_{G/G_\mathbb{Z}} \widehat{f} \d{\mu} = \frac{1}{\nu_\Gamma} \sum_{j = 1}^{h} \omega_j \widetilde{\lambda}(\kappa_jB).
    \]
    % where $\omega_j$ is defined in \cite[(2.3.2), (2.3.3)]{shucheng18}.
    Using the identification of $\mathscr{L} = G/Q$ with $K/M \times A$ and the explicit description of $\widetilde{\lambda}$ from \S\ref{secmain}, one can see that $\widetilde{\lambda}(\kappa_jB) = \kappa_j^n\widetilde{\lambda}(B)$.
    Thus
    \[
        \int_{G/G_\mathbb{Z}} \widehat{f} \d{\mu} = \frac{1}{\nu_\Gamma} \sum_{j = 1}^{h} \omega_j\kappa_j^n \widetilde{\lambda}(B) = \kappa \widetilde{\lambda}(B),
    \]
    where $\kappa = \frac{1}{\nu_\Gamma} \sum_{j = 1}^{h} \omega_j\kappa_j^n > 0$.
    % Since
    % \[
    %     \widehat{\mychi_{B}} = (\widehat{\mychi_{B}})_{\mathrm{pr}} + (\widehat{\mychi_{\frac{1}{2}B}})_{\mathrm{pr}} + (\widehat{\mychi_{\frac{1}{3}B}})_{\mathrm{pr}} + \cdots,
    % \]
    % we see that
    % \[
    %     \int_{G/G_\mathbb{Z}} \widehat{f} \d{\mu} = \kappa\zeta(n) \widetilde{\lambda}(B).
    % \]
    Finally, we find that
    \[
        \int_{G/\Gamma} \widehat{f} \d{\mu} = \frac{1}{[\Gamma : G_\mathbb{Z}]} \int_{G/G_\mathbb{Z}} \widehat{f} \d{\mu} = \frac{\kappa}{[\Gamma : G_\mathbb{Z}]} \widetilde{\lambda}(B) = \int_{\mathscr{L}} f \d{\lambda},
    \]
    where $\lambda = \dfrac{\kappa}{[\Gamma : G_\mathbb{Z}]} \widetilde{\lambda}$.
    Using standard measure theoretic techniques one can see that the above theorem holds for all $f \in L^1(\mathscr{L}, \lambda)$.

\end{proof}% F}O}L}D

\vspace{1.5em}

% UnnumberedSubsection
\subsection*{Some estimates about $F_{T, c}$:}% F{O{L{D
For $\bm{x} \in \mathscr{L}$ write $[\bm{x}] = x_{n+2} + x_{n+1}$.
Then note that $[g_t\bm{x}] = e^{-t}[\bm{x}]$.
Let $f_{r, c}$ be the characteristic function of $F_{r, c}$, then we have
\[
    f_{r, c}(g_t\bm{x}) =
    \begin{dcases}
        1 & \text{if} \ e^t \leq [\bm{x}] < e^{t+r} \ \text{and} \ x_{n+2}^{2} - x_{n+1}^{2} < c^2, \\
        0 & \text{otherwise.} \ 
    \end{dcases}
\]
Therefore, for $T > r$
% UnnumberedGather
\begin{gather*}
    \bm{x} \in F_{T, c} \implies |\{t \in [0, T] : g_t\bm{x} \in F_{r, c}\}| \leq r, \\
    \bm{x} \in F_{T, c} \smallsetminus F_{r, c} \implies |\{t \in [0, T] : g_t\bm{x} \in F_{r, c}\}| = r,
\end{gather*}
and
\[
    g_t\bm{x} \in F_{r, c} \ \text{for some} \ t \in [0, T] \implies \bm{x} \in F_{T+r, c}.
\]
Using~\eqref{eqsiegeltrans} and changing the order of summation and integration it follows that for any $\Lambda \in X$ and any $T > r$
% Equation
\begin{equation}\label{eqsandwichFfhat1}
    \#(\Lambda \cap (F_{T, c} \smallsetminus F_{r, c})) \leq \frac{1}{r} \int_{0}^{T} \widehat{f}_{r, c}(g_t\Lambda) \d{t} \leq \#(\Lambda \cap F_{T+r, c}).
\end{equation}
\eqref{eqsandwichFfhat1} further implies
% Equation
\begin{equation}\label{eqsandwichFfhat2}
    \frac{1}{r} \int_{0}^{T-r} \widehat{f}_{r, c}(g_t\Lambda) \d{t} \leq \#(F_{T, c}(\Lambda)) \leq \frac{1}{r} \int_{0}^{T} \widehat{f}_{r, c}(g_t\Lambda) \d{t} + \#(F_{r, c}(\Lambda)).
\end{equation}

Setting $f_{r, c, A}$ to be the characteristic function of $F_{T, c, A}$, by similar arguments we see that for any $\Lambda \in X$
% Equation
\begin{equation}\label{eqsandwichFfhatA}
    \frac{1}{r} \int_{0}^{T-r} \widehat{f}_{r, c, A}(g_t\Lambda) \d{t} \leq \#(F_{T, c, A}(\Lambda)) \leq \frac{1}{r} \int_{0}^{T} \widehat{f}_{r, c, A}(g_t\Lambda) \d{t} + \#(F_{r, c, A}(\Lambda)).
\end{equation}

% F}O}L}D

% StatementofTheorem
\begin{theorem}[Application of Moore's ergodicity theorem; \cite{moore66}, Theorem 3]\label{thmmoore}% F{O{L{D
    The action of $\{g_t\}$ on $X$ is ergodic with respect to the Haar measure $\mu$.
\end{theorem}% F}O}L}D

% StatementofTheorem
\begin{theorem}[Birkhoff's ergodic theorem]\label{thmbirkhoff}% F{O{L{D
    Let $\{g_t\}$ be an ergodic measure-preserving action on a probability space $(\Omega, \mu)$ and $f \in L^1(\Omega)$.
    Then for almost every $x \in \Omega$, we have
    % Equation
    \begin{equation}\label{eqbet}
        \lim_{T \to \infty} \frac{1}{T} \int_{0}^{T} f(g_t x) \d{t} = \int_{\Omega} f \d{\mu}.
    \end{equation}
\end{theorem}% F}O}L}D

% Remark
\begin{remark}\label{rembirkhoffgen}% F{O{L{D
    $x \in \Omega$ satisfying~\eqref{eqbet} is called \emph{Birkhoff generic} with respect to $f$.
    A proof of the above theorem could be found in \cite{katznelsonweiss82}.
\end{remark}% F}O}L}D

Hence applying Theorems~\ref{thmmoore} and~\ref{thmbirkhoff} to~\eqref{eqsandwichFfhat2} and using Theorem~\ref{thmsmvt}, we get that for almost every $\Lambda \in X$
% Equation
\begin{equation}\label{eqschforFandaelat}
    \lim_{T \to \infty} \frac{\#(F_{T, c}(\Lambda))}{T} = \lim_{T \to \infty} \frac{1}{Tr} \int_{0}^{T} \widehat{f}_{r, c}(g_t\Lambda) \d{t} = \frac{1}{r} \int_{X} \widehat{f}_{r, c} \d{\mu} = \frac{1}{r} \int_{\mathscr{L}} f_{r, c} \d{\lambda} = \frac{1}{r} |F_{r, c}|.
\end{equation}
\eqref{eqFTcAvol} shows that~\eqref{eqschforFandaelat} proves Theorem~\ref{thmschforlat}.

% Remark
\begin{remark}\label{}% F{O{L{D
    Any $\Lambda$ satisfying~\eqref{eqschforFandaelat} (or~\eqref{eqschforlat}) is Birkhoff generic with respect to $\widehat{f}_{r, c}$.
\end{remark}% F}O}L}D

% Proof
\begin{proof}[Proof of Corollary~\ref{corETcasympforlat}]% F{O{L{D
For any $\Lambda$ satisfying~\eqref{eqschforlat}, using~\eqref{eqETcAsandwiched} we see that for all sufficiently large $\ell$
% Equation
\begin{equation}\label{eqETclambdasandwiched}
    F_{T - r_0, c_\ell}(\Lambda) \smallsetminus (C_\ell \cap \Lambda) \subseteq E_{T, c}(\Lambda) \smallsetminus (C_0 \cap \Lambda) \subseteq F_{T + r_0, c}(\Lambda).
\end{equation}
Since $C_0$ and $C_\ell$ are compact $\#(C_0 \cap \Lambda), \#(C_\ell \cap \Lambda) < \infty$.
Similar arguments as in \S\ref{ssecasympETcvol} and~\eqref{eqETcAasymp} then proves Corollary~\ref{corETcasympforlat}.

% Then from~\eqref{eqschforFandaelat} and~\eqref{eqETclambdasandwiched} it follows that
% % Equation
% \begin{equation}\label{eqETclambdals}
%     \limsup_{T \to \infty} \frac{\#(E_{T, c}(\Lambda))}{T} \leq \limsup_{T \to \infty} \frac{\#(F_{T, c}(\Lambda))}{T} = \frac{1}{r}|F_{r, c}|.
% \end{equation}
% Also from~\eqref{eqschforFandaelat} and~\eqref{eqETclambdasandwiched} it follows that for all integers $\ell$ sufficiently large
% % Equation
% \begin{equation}\label{eqETclambdali1}
%     \liminf_{T \to \infty} \frac{\#(E_{T, c}(\Lambda))}{T} \geq \liminf_{T \to \infty} \frac{\#(F_{T - r_0, c_\ell}(\Lambda))}{T} = \frac{1}{r}|F_{r, c_\ell}|.
% \end{equation}
% Since $c \mapsto \frac{1}{r}|F_{r, c}|$ is a continuous function and $c_\ell \uparrow c$,~\eqref{eqETclambdali1} implies that
% % Equation
% \begin{equation}\label{eqETclambdali2}
%     \liminf_{T \to \infty} \frac{\#(E_{T, c}(\Lambda))}{T} \geq \frac{1}{r}|F_{r, c}|.
% \end{equation}
% From~\eqref{eqFTcAvol},~\eqref{eqETcAasymp},~\eqref{eqETclambdals} and~\eqref{eqETclambdali2}, Corollary~\ref{corETcasympforlat} follows.

\end{proof}% F}O}L}D

% F}O}L}D
% Section
\section{Proof of Theorems~\ref{thmschforsn} and~\ref{thmschforKlat}}\label{}% F{O{L{D
Let $U \subseteq N, I \subseteq A, V \subseteq K$ be open sets of finite measure with respect to measures $\d\bm{y}, \d{t}$ and $\sigma$ on $N, A$ and $K$ respectively; further assume that $U$ and $I$ are neighborhoods of identity.
Denote $UIV$ by $W$, which is open and has $\mu$-finite measure.
The following proposition is adapted from \cite[Proposition 4.2]{athreyaparrishtseng16}.

% Proposition
\begin{proposition}\label{propfubinitype}% F{O{L{D
    Let $W$ be as in the previous paragraph.
    Then for a.e.\ $u_{\bm{y}} \in U$ and a.e.\ $g_t \in I$ there exists a measurable subset $V_{\bm{y}, t} \subseteq V$ such that $\sigma(V_{\bm{y}, t}) = \sigma(V)$ and for every $k \in V_{\bm{y}, t}$, the lattice $u_{\bm{y}}g_t k\Delta = u_{\bm{y}}g_t\Delta_k$ is Birkhoff generic with respect to $\widehat{f}_{r, c}$.
\end{proposition}% F}O}L}D
% Proof
\begin{proof}% F{O{L{D
    Since $\mu$-a.e.\ element in $X$ is Birkhoff generic with respect to $\widehat{f}_{r, c}$, there exists a set $W_{0} \subseteq W$ such that $g\Delta$ is Birkhoff generic with respect to $\widehat{f}_{r, c}$ for every element $g \in W_0$, and $\mu(W_0) = \mu(W)$.
    For $u_{\bm{y}} \in U$ and $g_t \in I$ define $V_{\bm{y}, t} := \{k \in V : u_{\bm{y}}g_t k \in W_0\}$.
    Then Fubini's theorem implies that $V_{\bm{y}, t}$ is measurable for a.e.\ $u_{\bm{y}} \in U$ and a.e.\ $g_t \in I$.

    We claim that $\sigma(V_{\bm{y}, t}) = \sigma(V)$ for a.e.\ $u_{\bm{y}} \in U$ and a.e.\ $g_t \in I$.
    If not, let $S \subseteq U \times I$ be a positive measure set in $U \times I$ such that for all $(u_{\bm{y}}, g_t) \in S$, $\sigma(V_{\bm{y}, t}) < \sigma(V)$.
    Then integrating using Fubini's theorem
    % UnnumberedAlign
    \begin{align*}
        \nu_\Gamma \cdot \mu(W_0) &= \int_{U \times I} \int_{V} e^{-nt} \mathbbm{1}_{W_0}(u_{\bm{y}}g_t k) \d{\bm{y}}\d{t}\d{\sigma(k)} \\
        &= \int_{U \times I} \int_{V} e^{-nt} \mathbbm{1}_{V_{\bm{y}, t}}(k) \d{\bm{y}}\d{t}\d{\sigma(k)} \\
        &= \int\limits_{U \times I \smallsetminus S} \int_{V} e^{-nt} \mathbbm{1}_{V_{\bm{y}, t}}(k) \d{\bm{y}}\d{t}\d{\sigma(k)} + \int_{S} \int_{V} e^{-nt} \mathbbm{1}_{V_{\bm{y}, t}}(k) \d{\bm{y}}\d{t}\d{\sigma(k)} \\
        &< \int\limits_{U \times I \smallsetminus S} \int_{V} e^{-nt} \mathbbm{1}_{V}(k) \d{\bm{y}}\d{t}\d{\sigma(k)} + \int_{S} \int_{V} e^{-nt} \mathbbm{1}_{V}(k) \d{\bm{y}}\d{t}\d{\sigma(k)} \\
        &= \int_{U \times I} \int_{V} e^{-nt} \d{\bm{y}}\d{t}\d{\sigma(k)} \\
        &= \int_{W} e^{-nt} \d{\bm{y}}\d{t}\d{\sigma(k)} \\
        &= \nu_\Gamma \cdot \mu(W),
    \end{align*}
    a contradiction.
\end{proof}% F}O}L}D

To descend from a.e.\ lattices in $X$ to a.e.\ element in $\mathbb{S}^n$ we invoke the Iwasawa decomposition of $G$.
With $V = K$ in Proposition~\ref{propfubinitype} let $S' := \{(u_{\bm{y}}, g_s) \in U \times I : \sigma(K_{\bm{y}, s}) = \sigma(K) = 1\}$, then $S'$ is full measure in $U \times I$.
For $\bm{x} \in \mathscr{L}$ denote $(x_1, \ldots, x_{n})$ by $\widetilde{\bm{x}}$.
\newline

Let ${\{\varepsilon_{\ell}\}}_{\ell \in \mathbb{Z}_+} \to 0$ be a sequence of positive reals.
For each $\ell$ we are going to choose $(u_{\bm{y}_\ell}, g_{s_\ell}) \in S'$ satisfying the following conditions:
% Enumerate
\begin{enumerate}[label= (\roman*),font=\normalfont,before=\normalfont, leftmargin=2em]
    \item $\bm{y}_\ell \to \bm{0}$ and $s_\ell \downarrow 0$ as $\ell \to \infty$.
        For each $\ell$ there exists $K_\ell \subseteq K$ such that $\sigma(K_\ell) = 1$ and for all $k \in K_\ell$ the lattice $u_{\bm{y}_\ell}g_{s_\ell} \Delta_k$ is Birkhoff generic with respect to $\widehat{f}_{r, c-\varepsilon_\ell}$ and $\widehat{f}_{r, c+\varepsilon_\ell}$.

        Let $K_\infty = \bigcap K_\ell$, then $\sigma(K_\infty) = 1$.
        Fix $k \in K_\infty$, then $\Delta^{(\ell)} := u_{\bm{y}_\ell}g_{s_\ell}\Delta_k$ is Birkhoff generic with respect to $\widehat{f}_{r, c-\varepsilon_\ell}$ and $\widehat{f}_{r, c+\varepsilon_\ell}$ for all $\ell \in \mathbb{N}$.

        Now we are going to choose the `speed' at which $\bm{y}_\ell \to \bm{0}$ and $s_\ell \downarrow 0$.
    \item Let $F_c := \{\bm{x} \in \mathscr{L} : x_{n+2}^2 - x_{n+1}^2 < c^2, 1 \leq x_{n+2} + x_{n+1}\}$.
        Then $F_c \cap \Delta_k$ and $u_{\bm{y}_\ell}g_{s_\ell}F_c \cap \Delta^{(\ell)}$ naturally correspond to each other (under multiplication by $u_{\bm{y}_\ell} g_{s_\ell}$).
        For $\bm{x} \in g_{s_\ell}F_{c}$, $\widetilde{\bm{x}}$ is uniformly bounded and note that
        % UnnumberedGather
        \begin{gather*}
            {(u_{\bm{y}_\ell}\bm{x})}_{n+2} - {(u_{\bm{y}_\ell}\bm{x})}_{n+1} = x_{n+2} - x_{n+1} \\
            {(u_{\bm{y}_\ell}\bm{x})}_{n+2} + {(u_{\bm{y}_\ell}\bm{x})}_{n+1} = (1 + \norm{\bm{y}_\ell}^2)x_{n+2} + (1 - \norm{\bm{y}_\ell}^2)x_{n+1} + 2\braket{\widetilde{\bm{x}}}{\bm{y}_\ell}.
        \end{gather*}
        Since $\widetilde{\bm{x}}$ is uniformly bounded on $g_{s_\ell}F_{c}$, we can choose $\bm{y}_\ell$ so close to $\bm{0}$ that
        % Equation
        \begin{equation}\label{equylboxxest}
            (1 - \varepsilon_\ell')(x_{n+2} + x_{n+1}) < {(u_{\bm{y}_\ell}\bm{x})}_{n+2} + {(u_{\bm{y}_\ell}\bm{x})}_{n+1} < (1 + \varepsilon_\ell')(x_{n+2} + x_{n+1}),
        \end{equation}
        where $\varepsilon_\ell' \downarrow 0$ and $s_\ell \downarrow 0$ satisfies $e^{s_\ell} < 1 + \varepsilon_\ell', (1 + \varepsilon_\ell')^2 c < c + \varepsilon_\ell, (1 + \varepsilon_\ell')^2(c - \varepsilon_\ell) < c$ and $\varepsilon_\ell' < 1/2$ for all $\ell$.
\end{enumerate}

\eqref{equylboxxest} implies that $u_{\bm{y}_\ell} g_{s_\ell}F_{c}$ can be approximated from inside by $F_{c-\varepsilon_\ell}$ and from outside by $F_{c+\varepsilon_\ell}$, possibly up to two precompact sets $\mathcal{D}_1^{(\ell)}$ and $\mathcal{D}_2^{(\ell)}$.
The set $\mathcal{D}_2^{(\ell)}$ appears as follows: For  $u_{\bm{y}_\ell}g_{s_\ell}F_{c}$ there might exist points $\bm{x} \in F_{c}$ such that $u_{\bm{y}_\ell}g_{s_\ell}\bm{x} \in P$, where $P := \{\bm{x} \in \mathscr{L} : x_{n+2} + x_{n+1} \leq 1\}$.
Let
\[
    \mathcal{D}_2^{(\ell)} := \{\bm{x} \in F_{c} : u_{\bm{y}_\ell}g_{s_\ell}\bm{x} \in P\} = g_{-s_\ell}u_{-\bm{y}_\ell}(P) \cap F_c
\]
and
\[
    \mathcal{D}_1^{(\ell)} := \{\bm{x} \in F_{c-\varepsilon_\ell} : g_{-s_\ell}u_{-\bm{y}_\ell}\bm{x} \in \mathcal{D}_2^{(\ell)} \cup P\} = u_{\bm{y}_\ell}g_{s_\ell}(\mathcal{D}_2^{(\ell)} \cup P) \cap F_{c-\varepsilon_\ell}.
\]
then~\eqref{equylboxxest} implies $\mathcal{D}_1^{(\ell)}$ and $\mathcal{D}_2^{(\ell)}$ are bounded, and we have
\[
    F_{c-\varepsilon_\ell} \smallsetminus \mathcal{D}_1^{(\ell)} \subseteq u_{\bm{y}_\ell}g_{s_\ell}(F_{c} \smallsetminus \mathcal{D}_2^{(\ell)}) \subseteq F_{c+\varepsilon_\ell}.
\]
By similar arguments and using~\eqref{equylboxxest} we get
\[
    F_{T-1, c-\varepsilon_\ell} \smallsetminus \mathcal{D}_1^{(\ell)} \subseteq u_{\bm{y}_\ell}g_{s_\ell}(F_{T, c} \smallsetminus \mathcal{D}_2^{(\ell)}) \subseteq F_{T+1, c+\varepsilon_\ell},
\]
for all $\ell$ and $T$ sufficiently large.

Therefore
% Equation
\begin{equation}\label{equylboxxnoest}
    \#(F_{T-1, c-\varepsilon_\ell}(\Delta^{(\ell)})) - D_1^{(\ell)} \leq \#(F_{T, c}(\Delta_k)) - D_2^{(\ell)} \leq \#(F_{T+1, c+\varepsilon_\ell}(\Delta^{(\ell)})),
\end{equation}
where $D_1^{(\ell)} := \#(\mathcal{D}_1^{(\ell)} \cap \Delta^{(\ell)})$ and $D_2^{(\ell)} := \#(\mathcal{D}_2^{(\ell)} \cap \Delta_k)$.
Since a precompact set in $\mathscr{L}$ can only have a finite number of lattice points, it follows that $D_1^{(\ell)}, D_2^{(\ell)} < \infty$.
Consequently, from~\eqref{equylboxxnoest} we get
\[
    \lim_{T \to \infty} \frac{\#(F_{T-1, c-\varepsilon_\ell}(\Delta^{(\ell)}))}{T-1} \leq \lim_{T \to \infty} \frac{\#(F_{T, c}(\Delta_k))}{T} \leq \lim_{T \to \infty} \frac{\#(F_{T+1, c+\varepsilon_\ell}(\Delta^{(\ell)}))}{T+1}.
\]
Since $\Delta^{(\ell)}$ is Birkhoff generic with respect to $\widehat{f}_{r, c - \varepsilon_\ell}$ and $\widehat{f}_{r, c+\varepsilon_\ell}$ for all $\ell$, using~\eqref{eqschforFandaelat} we see that
\[
    \frac{1}{r}|F_{r, c-\varepsilon_\ell}| \leq \lim_{T \to \infty} \frac{\#(F_{T, c}(\Delta_k))}{T} \leq \frac{1}{r}|F_{r, c+\varepsilon_\ell}|.
\]
Since $c \mapsto |F_{r, c}|$ is a continuous function, letting $\ell \to \infty$ we see that $\forall k \in K_\infty$
\[
    \lim_{T \to \infty}\frac{\#(F_{T, c}(\Delta_k))}{T} = \frac{1}{r} |F_{r, c}|,
\]
Thus using~\eqref{eqFTcAvol} we have for $\sigma$-a.e.\ $k \in K$
% Equation
\begin{equation}\label{eqFTcdeltaksim}
    \#(F_{T, c}(\Delta_k)) \thicksim |F_{T, c}| \ \text{as} \ T \to \infty.
\end{equation}
Now using Corollary~\ref{corETcasympforlat} we get for $\sigma$-a.e.\ $k \in K$
\[
    \#(E_{T, c}(\Delta_k)) \thicksim |E_{T, c}| \ \text{as} \ T \to \infty.
\]
This completes proof of Theorem~\ref{thmschforKlat}.

Since for $\alpha \in \mathbb{S}^n$, $N_{T, c}(\alpha; \Delta) = \#(E_{T, c}(\Delta_k))$ with $k\alpha = \bm{u}_{n+1}$, and full measure subsets of $K$ naturally correspond to full measure subsets of $\mathbb{S}^n$, we have Theorem~\ref{thmschforsn}.

% F}O}L}D
% Section
\section{Proof of Theorems~\ref{thmspiralingforsn} and~\ref{thmspiralingwithFE}}\label{}% F{O{L{D
Since $A$ has boundary measure zero, $F_{r, c, A}$ also has boundary measure zero, and hence $f_{r, c, A}$ is Riemann integrable on $\mathscr{L}$.
As equation \eqref{eqFTcdeltaksim} holds for a.e.\ $k \in K$, using \eqref{eqsandwichFfhat1} we see that it is equivalent to
% Equation
\begin{equation}\label{eqfrckequi}
    \frac{1}{T} \int_{0}^{T} \widehat{f}_{r, c} (g_t\Delta_k) \d{t} \xrightarrow{\text{as} \ T \to \infty} \int_{X} \widehat{f}_{r, c} \d{\mu} \quad \text{for a.e.} \ k \in K.
\end{equation}
Now we will prove a statement similar to~\eqref{eqfrckequi} for $f_{r, c, A}$ for a.e.\ $k \in K$.

% Lemma
\begin{lemma}\label{lemfrcAhataecont}% F{O{L{D
    $\widehat{f}_{r, c, A}$ is continuous a.e.\ on $X$.
\end{lemma}% F}O}L}D
% Proof
\begin{proof}% F{O{L{D
    Let $S$ be the set of discontinuities of $f_{r, c, A}$ in $\mathscr{L}$.
    Then $|S| = 0$ and it follows that the set $S'$ of discontinuities of $\widehat{f}_{r, c, A}$ is contained in $S'' := \{\Lambda : \Lambda \cap S \neq \varnothing\}$.
    For each $\bm{v} \in \Lambda_0$, the set of $g \in G$ such that $g\bm{v} \in S$ has Haar measure zero in $G$, and hence $S''$ is measure zero.
    Therefore $\mu(S') = 0$.
\end{proof}% F}O}L}D

Using the fact that $N$ is the contracting horospherical subgroup associated to $\{g_t\}$ one can show that:
% Theorem
\begin{theorem}\label{thmccgen}% F{O{L{D
    Fix $\Lambda \in X$.
    Then for almost every $k \in K$ and for all $\varphi \in C_c(X)$
    \[
        \frac{1}{T} \int_{0}^{T} \varphi(g_tk\Lambda) \d{t} \to \int_{X} \varphi \d{\mu}.
    \]
\end{theorem}% F}O}L}D
% Proof
\begin{proof}% F{O{L{D
    % One can use the Birkhoff genericity of $u_{\bm{y}}g_sk\Lambda$ with respect to $C_c(X)$ functions and the uniform continuity of $\varphi$ to conclude this result.
    % This result is similar to the equidistribution result of \cite[\S1.3]{kleinbockmargulis96}.
    % We leave out the details.

    Since $\mu$ is ergodic under the $g_t$-action, one has that for almost every $\Delta \in X$ and for all $\varphi \in C_c(X)$
    \[
        \frac{1}{T} \int_{0}^{T} \varphi(g_t\Delta) \d{t} \to \int_{X} \varphi \d{\mu}.
    \]
    The above can be shown by compactifying $X$ and using, e.g., \cite[Lemma 6.3]{einsiedlerward11}.
    Any $\varphi \in C_c(X)$ is `uniformly continuous', i.e., for given $\varepsilon > 0$ there exists an open neighborhood $U$ of identity in $G$ such that for all $h \in U$ and $\Delta \in X$
    % Equation
    \begin{equation}\label{eqvphunif}
        |\varphi(h\Delta) - \varphi(\Delta)| < \varepsilon.
    \end{equation}
    Then since $g_tu_{\bm{y}}g_{-t} \to 1_G$ as $t \to \infty$, for fixed $\Lambda \in X$ we have
    % UnnumberedAlign
    \begin{align*}
        \lim_{T \to \infty} \frac{1}{T}\int_{0}^{T} \varphi(g_tu_{\bm{y}}g_sk\Lambda) \d{t} &= \lim_{T \to \infty} \frac{1}{T}\int_{0}^{T} \varphi(g_tu_{\bm{y}}g_{-t}g_{t+s}k\Lambda) \d{t}, \\
                                                                                            &= \lim_{T \to \infty} \frac{1}{T} \int_{0}^{T} \varphi(g_{t+s}k\Lambda) \d{t}, \quad \text{using \eqref{eqvphunif}}\\
                                                                                            &= \lim_{T \to \infty} \frac{1}{T} \int_{0}^{T} \varphi(g_{t}k\Lambda) \d{t}.
    \end{align*}
    
    Hence, arguments similar to Proposition \ref{propfubinitype} finish the proof.
\end{proof}% F}O}L}D
% Remark
\begin{remark}\label{}% F{O{L{D
    The above theorem seems to be well known to experts in the area, but we could not find a proper reference for it, that's why we included a short proof.
\end{remark}% F}O}L}D

We prove some general properties of convergence of measures as in~\cite{kleinbockshiweiss17} (Lemma 5.2 to Corollary 5.4).%, and also in~\cite{alamghosh20} (Lemma 5.5 to Corollary 5.7).

% Lemma
\begin{lemma}\label{lemconvmeasi}% F{O{L{D
    Let $\{\mu_i\}$ be a sequence of probability measures on $X$ such that $\mu_i \to \mu$ with respect to the weak-$\ast$ topology.
    Then for any non-negative $\varphi \in C_c(X)$ we have
    \[
        \lim_{i \to \infty} \int_{X} \varphi \cdot \widehat{f}_{r, c, A} \, \d{\mu_i} = \int_{X} \varphi \cdot \widehat{f}_{r, c, A} \, \d{\mu}.
    \]
\end{lemma}% F}O}L}D
% Proof
\begin{proof}% F{O{L{D
    Let $h$ be a compactly supported continuous function on $\mathscr{L}$ so that $f_{r, c, A} \leq h$.
    Therefore $\widehat{f}_{r, c, A} \leq \widehat{h}$, and hence $\widehat{f}_{r, c, A}$ is bounded on compact sets of $X$ since $\widehat{h}$ is continuous on $X$.
    Thus we see that $\varphi \cdot \widehat{f}_{r, c, A}$ is bounded, compactly supported and continuous except on a set of measure zero.
    By using a partition of unity, without loss of generality one can assume that $\varphi$ is supported on a coordinate chart.
    Applying Lebesgue's criterion for Riemann integrability to $\varphi \cdot \widehat{f}_{r, c, A}$, we can write $\int_{X} \varphi \cdot \widehat{f}_{r, c, A} \, \d{\mu}$ as the limit of upper and lower Riemann sums.
    It follows that given $\varepsilon > 0$ there exist $h_1, h_2 \in C_c(X)$ such that $h_1 \leq \varphi \cdot \widehat{f}_{r, c, A} \leq h_2$ and
    % Equation
    \begin{equation}\label{eqlemconvmeasi}
        \int_{X} (h_2 - h_1) \, \d{\mu} \leq \varepsilon.
    \end{equation}
    Thus we have
    % Gather
    \begin{gather}
        \int_{X} h_1 \, \d{\mu} \leq \liminf_{i \to \infty} \int_{X} \varphi \cdot \widehat{f}_{r, c, A} \, \d{\mu_i} \leq \limsup_{i \to \infty} \int_{X} \varphi \cdot \widehat{f}_{r, c, A} \, \d{\mu_i} \leq \int_{X} h_2 \, \d{\mu} \label{eqlemconvmeasii} \\
        \int_{X} h_1 \, \d{\mu} \leq \int_{X} \varphi \cdot \widehat{f}_{r, c, A} \, \d{\mu} \leq \int_{X} h_2 \, \d{\mu}. \label{eqlemconvmeasiii}
    \end{gather}
    Since $\varepsilon$ was arbitrary, the lemma follows from~\eqref{eqlemconvmeasi} --~\eqref{eqlemconvmeasiii}.
\end{proof}% F}O}L}D

% Corollary
\begin{corollary}\label{corconvmeasi}% F{O{L{D
    Let the notation be as in Lemma~\ref{lemconvmeasi}.
    Assume that
    % Equation
    \begin{equation}\label{eqcorconvmeasi}
        \lim_{i \to \infty} \int_{X} \widehat{f}_{r, c} \, \d{\mu_i} = \int_{X} \widehat{f}_{r, c} \, \d{\mu}.
    \end{equation}
    Then for any $\varepsilon > 0$ there exists $i_0 > 0$ and $\varphi \in C_c(X)$ with $0 \leq \varphi \leq 1$ such that
    % Equation
    \begin{equation}\label{eqcorconvmeasii}
        \int_{X} (1 - \varphi) \widehat{f}_{r, c} \, \d{\mu_i} < \varepsilon
    \end{equation}
    for any $i \geq i_0$.
\end{corollary}% F}O}L}D
% Proof
\begin{proof}% F{O{L{D
    Since $\widehat{f}_{r, c} \in L^1(X, \mu)$, there exists a compactly supported continuous function $\varphi : X \to [0, 1]$ such that
    % Equation
    \begin{equation}\label{eqcorconvmeasiii}
        \int_{X} (1 - \varphi) \widehat{f}_{r, c} \, \d{\mu} < \frac{\varepsilon}{3}.
    \end{equation}
    By Lemma~\ref{lemconvmeasi} and~\eqref{eqcorconvmeasi}, there exists $i_0 > 0$ such that for $i \geq i_0$
    % Gather
    \begin{gather}
        {\left|\int_{X} \varphi \cdot \widehat{f}_{r, c} \, \d{\mu_i} - \int_{X} \varphi \cdot \widehat{f}_{r, c} \, \d{\mu}\right|} < \frac{\varepsilon}{3} \label{eqcorconvmeasiv} \\
        {\left|\int_{X} \widehat{f}_{r, c} \, \d{\mu_i} - \int_{X} \widehat{f}_{r, c} \, \d{\mu}\right|} < \frac{\varepsilon}{3}. \label{eqcorconvmeasv}
    \end{gather}
    Therefore the corollary follows from~\eqref{eqcorconvmeasiii} --~\eqref{eqcorconvmeasv}.
\end{proof}% F}O}L}D

% Corollary
\begin{corollary}\label{corconvmeasii}% F{O{L{D
    Let the notation be as in Lemma~\ref{lemconvmeasi}.
    Assume that
    % Equation
    \begin{equation}\label{eqcorconvmeasvi}
        \lim_{i \to \infty} \int_{X} \widehat{f}_{r, c} \, \d{\mu_i} = \int_{X} \widehat{f}_{r, c} \, \d{\mu}.
    \end{equation}
    Then
    \[
        \lim_{i \to \infty} \int_{X} \widehat{f}_{r, c, A} \, \d{\mu_i} = \int_{X} \widehat{f}_{r, c, A} \, \d{\mu}.
    \]
\end{corollary}% F}O}L}D
% Proof
\begin{proof}% F{O{L{D
    Using Lemma~\ref{lemconvmeasi}, Corollary~\ref{corconvmeasi} and~\eqref{eqcorconvmeasvi}, we have that for $\varepsilon > 0$ there exists $i_0 > 0$ and a continuous compactly supported function $\varphi : X \to [0, 1]$ such that for $i \geq i_0$
    % UnnumberedGather
    \begin{gather*}
        {\left|\int_{X} \varphi \cdot \widehat{f}_{r, c, A} \, \d{\mu_i} - \int_{X} \varphi \cdot \widehat{f}_{r, c, A} \, \d{\mu}\right|} < \frac{\varepsilon}{3}, \\
        \int_{X} (1 - \varphi) \widehat{f}_{r, c} \, \d{\mu_i} < \frac{\varepsilon}{3}, \\
        \int_{X}(1 - \varphi) \widehat{f}_{r, c} \, \d{\mu} < \frac{\varepsilon}{3}.
    \end{gather*}
    Using $0 \leq \widehat{f}_{r, c, A} \leq \widehat{f}_{r, c}$ and the above inequalities, we get that
    \[
        {\left|\int_{X} \widehat{f}_{r, c, A} \, \d{\mu_i} - \int_{X} \widehat{f}_{r, c, A} \, \d{\mu}\right|} < \varepsilon
    \]
    for $i > i_0$.
    Hence we are done.
\end{proof}% F}O}L}D

\vspace{1em}
Using~\eqref{eqfrckequi}, Corollary~\ref{corconvmeasii} and Theorem~\ref{thmccgen}, we see that for almost every $k \in K$
\[
    \frac{1}{T} \int_{0}^{T} \widehat{f}_{r, c, A}(g_t\Delta_k) \d{t} \xrightarrow{\text{as} \ T \to \infty} \int_{X} \widehat{f}_{r, c, A} \d{\mu} = \int_{\mathscr{L}} f_{r, c, A} \d{\lambda} = |F_{r, c, A}|.
\]
But~\eqref{eqsandwichFfhatA} says that
\[
    \lim_{T \to \infty} \frac{1}{Tr} \int_{0}^{T} \widehat{f}_{r, c, A}(g_t\Delta_k) \d{t} = \lim_{T \to \infty} \frac{\#(F_{T, c, A}(\Delta_k))}{T}.
\]
Therefore using~\eqref{eqFTcAvol} we see that for almost every $k \in K$
% Equation
\begin{equation}\label{eqspiralingwithF}
    \#(F_{T, c, A}(\Delta_k)) \thicksim |F_{T, c, A}| \ \text{as} \ T \to \infty.
\end{equation}
For any $k \in K$ satisfying~\eqref{eqspiralingwithF} using~\eqref{eqETcAsandwiched} we see that
\[
    F_{T-r_0, c_\ell, A}(\Delta_k) \smallsetminus (C_\ell \cap \Delta_k) \subseteq E_{T, c, A}(\Delta_k) \smallsetminus (C_0 \cap \Delta_k) \subseteq F_{T+r_0, c, A}(\Delta_k).
\]
Similar arguments as in \S\ref{ssecasympETcvol} and in Proof of Corollary~\ref{corETcasympforlat} at the end of \S\ref{secpfsforlats}, and~\eqref{eqETcAasymp} completes the proof of Theorem~\ref{thmspiralingwithFE}.

Since full measure subsets of $K$ naturally correspond to full measure subsets of $\mathbb{S}^n$, Theorem~\ref{thmspiralingwithFE} implies Theorem~\ref{thmspiralingforsn}.

% F}O}L}D

% Bibliography
\bibliographystyle{alpha}
\bibliography{dio-on-sn-v5}

\end{document}